\documentclass[11pt]{amsart}
\usepackage{amscd,amssymb}
\usepackage{amsmath, amssymb, verbatim,color, amscd}
\usepackage[mathscr]{eucal}
\usepackage{tikz,graphicx}

\input xy
\xyoption{all}

\newtheorem{theorem}{Theorem}[section]
\newtheorem{lemma}[theorem]{Lemma}

\newtheorem{corollary}[theorem]{Corollary}

\renewcommand{\geq}{\geqslant}

\theoremstyle{definition}

\theoremstyle{definition}

\numberwithin{equation}{section}

\numberwithin{equation}{section}
 \numberwithin{figure}{section}

\author{Yuguang Zhang}
\address{Institut f\"{u}r Differentialgeometrie,  Gottfried Wilhelm Leibniz Universit\"{a}t Hannover,  Welfengarten 1, 30167 Hannover, Deutschland}
\email{yuguang.zhang@math.uni-hannover.de}
\title[]{A Note on Lagrangian Submanifolds in Symplectic $4m$-Manifolds}
\begin{document}
\begin{abstract} This short paper shows a topological  obstruction of the existence of certain  Lagrangian submanifolds in symplectic $4m$-manifolds.
\end{abstract}

\maketitle

\section{Introduction}
Let $(X, \omega)$ be a  symplectic manifold. A submanifold $L\subset X$ is called Lagrangian if $\omega|_L\equiv0$ and $\dim L=\frac{1}{2}\dim X$.  See the  textbook and papers  \cite{MS,Wei,ALP} for basic properties, general  backgrounds, and developments in the study of   Lagrangian submanifolds.  One question is to understand  topological  obstructions of embedding manifolds into symplectic manifolds as Lagrangian submanifolds, which has been intensively studied  (cf. \cite{ALP,Bi,BC,CM,Kh,S,V,Wel,WZ}).

 In \cite{Z}, a phenomenon has been  noticed while  searching for  obstructions of the existence of positive scalar curvature  Riemannian  metrics   on symplectic 4-manifolds. More precisely,
 Theorem 2.3 in   \cite{Z} should be interpreted as a simple  topological consequence of the presence   of certain  Lagrangian submanifolds in symplectic 4-manifolds, and this  topological constraint is a well-known  obstruction of positive scalar curvature metrics by quoting Taubes's theorem on  the Seiberg-Witten invariant.  The goal of this paper is to reformulate this result in a more suitable context, and to generalise it to the case of $4m$-dimensional symplectic manifolds, which might have some independent interests.

If $X$ is a compact oriented $4m$-dimensional  manifold, the intersection pairing  is the symmetric non-degenerated   quadratic form $$Q:H^{2m}(X, \mathbb{R})\oplus H^{2m}(X, \mathbb{R})\rightarrow \mathbb{R}, \  \  \  \  \  \  \  Q([\alpha],[\beta])=\int_X \alpha\wedge\beta,$$ which is  represented by the diagonal  matrix diag$(1,\cdots,1,-1,\cdots,-1)$ under  a suitable basis. Denote $b_{2m}^+(X)$ the number of  plus ones  in the matrix, and $b_{2m}^-(X)$ the number of minus ones.  Both of   $b_{2m}^\pm(X)$ are topological invariants, and the  $2m$-Betti number $b_{2m}(X)=b_{2m}^+(X)+b_{2m}^-(X)$.  If $X$ admits a symplectic form $\omega$, then $[\omega^m]\in H^{2m}(X, \mathbb{R})$, and $$Q([\omega^m], [\omega^m])=\int_X \omega^{2m}>0,$$ which implies  $$ b_{2m}^+(X)\geq1.$$

The main theorem asserts  that $ b_{2m}^\pm(X)$ provides some constraints of  the existence of certain Lagrangian submanifolds.

\begin{theorem}\label{thm1} Let $(X, \omega)$ be a compact symplectic manifold of dimension $4m$,  $m\geq1$, $ m\in \mathbb{Z}$, and $L$ be an orientable  Lagrangian submanifold in $X$.
  \begin{enumerate}
\item[i)] If  the Euler characteristic $\chi(L)$ of $L$ satisfies $$(-1)^{m}\chi(L)> 0, \ \ \ \ \ \ \ then \  \  \  \  \  \  \  \  \  b_{2m}^+(X)\geq 2.$$
     \item[ii)] If  $L$   represents a non-trivial homology   class, i.e. $$ 0\neq [L]\in H_{2m}(X, \mathbb{R}), \ \ \ and \ \ \   \chi(L) =0, $$  $$then \  \  \  \  \  \  \  \  \  b_{2m}^+(X)\geq 2, \  \  \  \  and \  \  \ \  \ b_{2m}^-(X)\geq1.$$
         \item[iii)] If $$(-1)^{m}\chi(L)< 0, \ \ \ \ \ \ \ then \  \  \  \  \  \  \  \  \  b_{2m}^-(X)\geq 1.$$
      \item[iv)]  If the $2m$-Betti number $b_{2m}(X)=1$, then   $$\chi(L)=0, \ \  \ and  \  \  \  0= [L]\in H_{2m}(X, \mathbb{R}).$$\end{enumerate}
\end{theorem}

When $X$ is 4-dimensional, this theorem is known to experts (Proposition 2.17 in \cite{WZ}), and 
 i)-ii)  have appeared in the proof of Theorem 2.3 in   \cite{Z}.
 There are some immediate applications of Theorem \ref{thm1} that are surely known to experts (cf \cite{BC,Bi,CM,EGH,S,V}):
 \begin{enumerate}
\item[i)]
 Neither  $\mathbb{CP}^{m}$ nor $S^{2m}$ can be embedded in $\mathbb{CP}^{2m}$ as a Lagrangian submanifold.
 \item[ii)] Only possible orientable  Lagrangian submanifolds in  $\mathbb{CP}^{2}$ are 2-tori $T^2$. Orientable Lagrangian submanifolds in
 $\mathbb{CP}^{1}\times \mathbb{CP}^{1}$ are either tori $T^2$ or spheres $S^2$ by $ b_2^+(\mathbb{CP}^{1}\times \mathbb{CP}^{1})=1$.
  \item[iii)] There is no Lagrangian sphere $S^{4m}$ in  $\mathbb{CP}^{4m-1}\times \Sigma$, where $\Sigma$ is a compact Riemann surface, since $ b_{4m}^+(\mathbb{CP}^{4m-1}\times \Sigma)=1$.
\end{enumerate}

Theorems in \cite{CM,EGH,V} assert that     Riemannian manifolds with negative  sectional
curvature do not admit any  Lagrangian embedding  into   uniruled symplectic manifolds.
As a corollary, we show certain constraints of  the existence of  special metrics on  Lagrangian submanifolds.

\begin{corollary}\label{cor2} Let $(X, \omega)$ be  a compact symplectic $4m$-dimensional  manifold with  $$b_{2m}^+(X)=1.$$ \begin{enumerate}
\item[i)] A compact   orientable $2m$-manifold $Y$ admitting a hyperbolic metric, i.e. a Riemannian metric with sectional curvature $-1$, cannot be embedded into $X$ as a Lagrangian submanifold.   \item[ii)] If $m=2$, and
   $L$ is an orientable  Lagrangian submanifold in $X$ admitting an Einstein metric $g$, i.e. the Ricci tensor $${\rm Ric}(g)=\lambda g$$ for a constant $\lambda\in\mathbb{R}$,  then  a  finite covering  of $L$ is the 4-torus $T^4$.  \end{enumerate}
\end{corollary}

We remark that the  condition $b_{2m}^+(X)=1$ is necessary for Corollary \ref{cor2} ii). For example, any symplectic 4-manifold $(X, \omega)$ admitting an Einstein metric could be diagonally embedded into the symplectic
8-manifold $(X, \omega)\times (X, -\omega)$ as a Lagrangian.

We prove Theorem \ref{thm1} and Corollary \ref{cor2} by using classical differential topology/geometry techniques in the next section.

{\bf Acknowledgments.} The author thanks the referee for many helpful suggestions. 

\section{Proofs}
Firstly, we recall  basic relevant     facts, and provide the  detail proofs for the completeness.
Let $(X, \omega)$ be a compact symplectic $4m$-dimensional manifold, and $L$ be an orientable  Lagrangian submanifold in $X$.  The symplectic form $\omega$ induces a natural orientation of $X$, i.e. given by $\omega^{2m}$.   We fix an orientation on $L$.

\begin{lemma}\label{le1} $$[L]^2= (-1)^m \chi(L),$$ where $[L]^2$ is the self-intersection number of $L$ in $X$. Consequently, if $\chi(L) \neq0$, then $[L]^2\neq0$, and    $$ 0\neq [L]\in H_{2m}(X, \mathbb{R}).$$
\end{lemma}
\begin{proof}
 The Weinstein neighbourhood theorem (cf. \cite{Wei,MS})  asserts that there is a tubular  neighbourhood $U$ of $L$ in $X$ diffeomorphic to a  neighbourhood of the zero section in the cotangent bundle $T^*L$, and the restriction of $\omega$ on $U$  is the canonical symplectic form on  $T^*L$ under the identification. More precisely,
if $x_1, \cdots, x_{2m}$ are local coordinates on $L$, then the  symplectic form $$\omega=\sum_{i=1}^{2m} dx_i\wedge dp_i,$$ where $p_1, \cdots, p_{2m}$ are coordinates on fibres  induced by $dx_j$, i.e. any 1-form  $\sum\limits_{i=1}^{2m} p_idx_i$ is given by the numbers $p_j$. We regard $L$ as the zero section of $T^*L$, and assume that $dx_1\wedge\cdots\wedge dx_{2m}$ gives the orientation of $L$.

 If we denote  $\varpi_i= dx_i\wedge dp_i$,  $i=1, \cdots, 2m $,  then the algebraic relationships are  $$  \varpi_i\wedge \varpi_j=\varpi_j\wedge \varpi_i, \  \  \ i\neq j, \  \  \  \  \  {\rm and } \ \  \  \varpi_i^2=0.$$
   The orientation inherited from the symplectic manifold is given by $$\omega^{2m}=(2m)!\varpi_1\wedge\cdots \wedge  \varpi_{2m}.$$ The orientation on $L$ induces an orientation on $T^*L$ given by $$dx_1\wedge\cdots \wedge  dx_{2m}\wedge dp_1\wedge\cdots \wedge dp_{2m}=(-1)^{m} \varpi_1\wedge\cdots \wedge  \varpi_{2m}. $$  If $L'$ is a small deformation of $L$ intersecting with $L$ transversally, the intersection number $L\cdot L'$ with respect to the orientation given by $\omega^{2m}$ is the self-intersection number of $L$ in $X$, i.e. $L\cdot L'=[L]^2$.  The Euler characteristic of $L$ equals to the intersection number $L\tilde{\cdot} L'$ with respect to the orientation defined by $dx_1\wedge\cdots \wedge  dx_{2m}\wedge dp_1\wedge\cdots \wedge dp_{2m}$ (cf. \cite{BT}), i.e. $L\tilde{\cdot} L'= \chi(L)$,  where we  identify the tangent bundle $TL$ with $T^*L$ via a Riemannian metric. Since $ L\cdot L'=(-1)^mL\tilde{\cdot} L' $, we obtain
   $ [L]^2= (-1)^m \chi(L)$.\end{proof}

  There are analogous formulae in more general contexts, e.g. see \cite{Web}.

Let $g$ be a Riemannian metric  compatible with $\omega$, i.e. $g(\cdot, \cdot)=\omega(\cdot, J \cdot)$ for an almost complex structure $J$ compatible with $\omega$, and let $\ast$ be the Hodge star operator of $g$. The volume form $dv_g$ of $g$ is the symplectic volume form, i.e.  $dv_g= \frac{1}{(2m)!}\omega^{2m} $.  When $\ast$ acts on $\wedge^{2m}T^*X$, $\ast^2=$Id holds.   $H^{2m}(X, \mathbb{R})$ is isomorphic to  the space of harmonic $2m$-forms by  the Hodge theory   (cf. \cite{GH}).  Thus  $H^{2m}(X, \mathbb{R})$  admits the self-dual/anti-self-dual  decomposition  $$H^{2m}(X, \mathbb{R})\cong \mathcal{H}_{+}(X)\oplus \mathcal{H}_{-}(X), \  \  \  \  \ {\rm where}$$  $$\mathcal{H}_{\pm}(X)=\{\alpha \in C^\infty(\wedge^{2m}T^*X) | d\alpha=0,  \  \    \ast\alpha=\pm \alpha\}.$$ Note that $b_{2m}^\pm(X)=\dim  \mathcal{H}_{\pm}(X).$

\begin{lemma}\label{le2} $$\omega^m\in \mathcal{H}_{+}(X).$$
\end{lemma}
\begin{proof}    Since $d\omega^m=0$,
   we only need to verify  that $\omega^m$ is self-dual, i.e. $\ast\omega^m=\omega^m$, which is a point-wise condition. For any point $x\in X$, we choose coordinates $x_1, \cdots, x_{2m},p_1, \cdots, p_{2m}$ on a neighbourhood of $x$ such that on the tangent space $T_xX$, $$\omega=\sum_{i=1}^{2m} dx_i\wedge dp_i, \  \ \ {\rm and} \ \ \ g=\sum_{i=1}^{2m} (dx_i^2+dp_i^2).$$ The calculation is the same as those in the K\"{a}hler case (cf. \cite{GH}) because of the locality.   If we still write $\varpi_i= dx_i\wedge dp_i$,  $i=1, \cdots, 2m $, then the volume form $dv_g=\varpi_1\wedge\cdots \wedge  \varpi_{2m}.$
Denote the multi-indices sets  $I=\{i_1, \cdots, i_m\}$, $i_1 < \cdots < i_m$,  and the complement  $I^c=\{1, \cdots, 2m\}\backslash I$. Note that  $$\omega^m=(\varpi_{1}+\cdots+ \varpi_{2m})^m=m!\sum_{I\subset \{1, \cdots, 2m\} }\varpi_I, \  \ \     \varpi_I= \varpi_{i_1}\wedge\cdots\wedge \varpi_{i_m}.$$  By $ \varpi_I\wedge \ast \varpi_I= dv_g= \varpi_I\wedge  \varpi_{I^c}$, we obtain
$\ast \varpi_I= \varpi_{I^c}$. Hence $\ast\omega^m=\omega^m$.   \end{proof}

Now we are ready to prove the results.

\begin{proof}[Proof of Theorem \ref{thm1}]  Note that the assumptions imply $0\neq [L]\in H_{2m}(X, \mathbb{R})$ in the cases i), ii) and iii).     If $PU(L)\in H^{2m}(X, \mathbb{R})$ is the Poincar\'{e} dual of $[L]$, then $$PU(L)=\alpha_+ +\alpha_-\neq0,  \  \ \  \ {\rm and} \ \ \ \ \alpha_{\pm}\in \mathcal{H}_{\pm}(X).$$  Since $\int_X \alpha_+\wedge \alpha_- =0$,
     $$ (-1)^m \chi(L)= [L]^2=\int_X(\alpha_+ +\alpha_-)^2=\int_X \alpha_+^2+ \int_X \alpha_-^2.$$

If  $(-1)^m \chi(L) \geq0$,     then $$ \int_X \alpha_+^2 \geq - \int_X \alpha_-^2 = \int_X \alpha_- \wedge\ast\alpha_-=\int_X |\alpha_-|^2_gdv_g \geq 0.$$
If $\alpha_+ =0$, then $\alpha_-=0$ and $PU(L)=0$, which is a contradiction. Thus $\alpha_+\neq 0$, and $$\int_{L}\alpha_+=\int_X\alpha_+ \wedge (\alpha_++\alpha_-)=\int_X \alpha_+^2 =\int_X|\alpha_+|_g^2dv_g\neq 0.$$ Since $$\int_{L}\omega^m =0,$$ $\omega^m $ and $\alpha_+$ are linearly independent in $\mathcal{H}_{+}(X)$. Therefore, we obtain  $b_{2m}^+(X)\geq 2$.

If $\chi(L)=0$, then $$0= [L]^2=\int_X \alpha_+^2+ \int_X \alpha_-^2, \  \  \ {\rm and} \  \ \ \int_X |\alpha_-|^2_gdv_g=\int_X |\alpha_+|^2_gdv_g.$$  Thus $\alpha_-\neq0$  and $b_{2m}^-(X)\geq 1$.

  If we assume $(-1)^m \chi(L) <0$, then $$0> \int_X \alpha_+^2+ \int_X \alpha_-^2, \  \  \ {\rm and} \  \ \ \int_X |\alpha_-|^2_gdv_g>\int_X |\alpha_+|^2_gdv_g\geq0.$$ We obtain the conclusion $b_{2m}^-(X)\geq 1$.

If $b_{2m}(X)=1$, then  $b_{2m}^+(X)=1$,  $b_{2m}^-(X)=0$,  and $H^{2m}(X, \mathbb{R})\cong \mathcal{H}_{+}(X)$.    Hence  $\chi(L)=0$ and  $[L]=0$.
\end{proof}

\begin{proof}[Proof of Corollary \ref{cor2}]  Let $Y$ be an orientable compact   $2m$-manifold, and $h$ be a hyperbolic metric on $Y$. By  the Gauss-Bonnet formula of hyperbolic manifolds (cf. \cite{H,KZ}), $$(-1)^m\chi(Y)=\varepsilon_{2m}{\rm Vol}_h(Y)>0,$$ where $\varepsilon_{2m}>0$ is a constant depending only on $m$, and ${\rm Vol}_h(Y)$ is the volume of the hyperbolic metric. Theorem \ref{thm1} implies  i).

 Now we assume that $m=2$. If $g$ is a Riemannian  metric on $L$, then  the Gauss-Bonnet-Chern formula for 4-manifolds reads   $$\chi(L)= \frac{1}{8\pi^2}\int_{L}(\frac{R^2}{24}+|W^+|_g^2+|W^-|_g^2-\frac{1}{2}|{\rm Ric}^o|_g^2)dv_g,$$ where $R$ is the scalar curvature, $W^\pm$ denotes the self-dual/anti-self-dual Weyl curvature of $g$, and ${\rm Ric}^o={\rm Ric}-\frac{R}{4}g$ is the Einstein tensor (cf. \cite{B}).  If $g$ is Einstein, then  ${\rm Ric}^o\equiv 0$ and thus $\chi(L) \geq0$.
By    Theorem \ref{thm1}, either  $\chi(L)<0$,  or $\chi(L)=0$ and   $[L]=0$ in $H^{4}(X, \mathbb{R})$ since $b_{4}^+(X)=1 $.  Therefore $\chi(L)=0$, which implies that $$R\equiv0,  \  \  \  \  W^\pm\equiv0,  \  \  \  \  \  {\rm Ric}\equiv 0.$$  The curvature tensor of $g$ vanishes, i.e. $g$ is a flat metric. Thus a  finite covering is the torus $T^4$.
\end{proof}

\appendix
\section{ An alternative proof} $$  \text{ by Anonymous Referee}  $$

Proposition 2.17 of  \cite{WZ} gives a very simple argument using the light cone lemma (which is a version
of the Cauchy-Schwartz inequality) 
to eventually prove Theorem 1.1 (i)(ii),
the harder parts, for $m = 1$. In fact, using light cone lemma to give topological restriction on Lagrangians
is not new (cf.  \cite{Wel,Kh}). This argument is easily generalized to prove the  whole Theorem 1.1 for
arbitrary $m$ as follows.  

The light cone lemma asserts that 
for the light cone $ \mathcal{C}$ of signature $(1, n)$ ($n > 0$),  $ \mathcal{C} \subset \mathbb{R}^{1,n} $, any two elements in the forward cone $\mathcal{C}_+ $ have
non-negative inner product. Especially, if the inner product is zero then the two elements are
proportional to each other. Note that  $ \mathcal{C}\backslash \{0\}$  has two connected components. One is the  forward cone $\mathcal{C}_+ $, and the other is the backward cone   $\mathcal{C}_- $.  Furthermore $ \mathcal{C}_+=-\mathcal{C}_-$.

Now we follow Proposition 2.17 of  \cite{WZ} to prove Theorem \ref{thm1} instead.  If $ (-1)^m \chi(L)\geq0$, we have $$PU(L)^2\geq0, \  \  \  \   
[\omega^m]^2>0,  \  \  \  \  [\omega^m]\cdot[L]=0.$$ If we assume $[L]\neq 0$ and $b_{2m}^+(X)=1$, both classes $PU(L)$ and $[\omega^m]
\in H^{2m}(X, \mathbb{R})$ are in the forward cone by choosing a suitable orientation of $L$.   Then
the light cone lemma implies this is a contradiction. That is, if  $b_{2m}^+(X)=1$,  then $ (-1)^m \chi(L)<0$, and if $\chi(L)=0$ then $[L]=0\in  H_{2m}(X, \mathbb{R})$. Other parts of Theorem \ref{thm1} can be similarly argued and are easier (without using Riemannian metrics and self-dual/anti-self-dual decompositions).   

Finally, the same argument of Corollary \ref{cor2} (i) also works to show that a manifold
admitting a complex hyperbolic metric, $X\cong \mathbb{CH}^m/\Gamma$, cannot be embedded as
a Lagrangian since   $$ { \rm Vol}(X)=\frac{(-4\pi)^m}{(m+1)!} \chi(X)$$ by the traditional Gauss-Bonnet theorem.  This in particular
implies that fake projective planes cannot be embedded as Lagrangians of 
$8$-dimensional symplectic manifolds with $ b_4^+=1$.

\end{document}